\documentclass[english,11pt]{smfart}
\usepackage{amscd}
\usepackage{latexsym}
\usepackage{amssymb}
\usepackage[all]{xy}
\usepackage{eufrak}
\usepackage{euscript}
\usepackage{amscd}

\def\Nd{\EuScript{N}^1}

\def\oT{{[0{,}T]}}

\def\cal{\mathcal}

\renewcommand{\b}{\overline}

\newtheorem{thm}{Theorem}
\newtheorem{defi}{Definition}
\newtheorem{lem}{Lemma}
\newtheorem{pro}{Proposition}
\newtheorem{cor}{Corollary}

\newtheorem{rema}{Remark}
\def\cal{\mathcal}

\def\di{\displaystyle}

\def\sg{\sigma}

\newcommand{\N}{\mathbb{N}}

\newcommand{\R}{\mathbb{R}}

\newcommand{\C}{\mathbb{C}}

\newcommand{\bP}{\mathbb{P}}
\newcommand{\rL}{{\rm L}}

\usepackage{graphicx}
\usepackage{a4wide}

\begin{document}
\setcounter{tocdepth}{3}
\baselineskip 6mm
\title[The Stokes, Navier-Stokes and Euler equations]{Lagrangian structures for the Stokes, Navier-Stokes and Euler equations}
\author{Jacky Cresson}
\author{S\'ebastien Darses}

\maketitle

\vskip 8mm
\begin{abstract}
We prove that the Navier-Stokes, the Euler and the Stokes equations admit a Lagrangian structure using the stochastic embedding of Lagrangian systems. These equations coincide with extremals of an explicit stochastic Lagrangian functional, i.e. they are stochastic Lagrangian systems in the sense of \cite{cd}.
\end{abstract}

\begin{small}
\tableofcontents
\end{small}

\section{Introduction}

The current paper aims at finding out a Lagrangian structure for some partial differential equations including the Euler, the Stokes and the Navier-Stokes equations. By a Lagrangian partial differential equation we mean that we can be find an explicit Lagrangian function and an adapted variational calculus, for example the quantum, fractional or stochastic calculus of variation, such that the solutions of the PDE are critical points of the Lagrangian functional. This is the classical inverse problem for partial differential equations. 

The interest of such results is at least twofold:

- First of all, we now have an intrinsic object, the Lagrangian, which control the dynamical behaviour of the PDE. 

- Secondly, this intrinsic structure can be used to derive more adapted numerical schemes for these equations. This will be discussed in a forthcoming paper.\\

However, classical results in the study of incompressible fluids ensure us that there exists obstructions to obtain a variational principle from which the behaviour of the fluid derives. An analogous result exists for the derivation of a Lagrangian structure for non conservative mechanical systems (see Bateman \cite{bate}). However, these obstructions results are based on the fundamental assumption that we use the classical differential calculus. As a consequence, these obstructions may disappear if we use a different "differential" calculus by extending the underlying class of objects to stochastic processes for instance. 

Stochastic interpretations of the Navier-Stokes equations already exit in a huge literature. For instance, C. peskin gives in \cite{peskin} a random-walk interpretation based on a model where the force exerted by a molecule of the fluid is considered. In \cite{busnello} B. Busnello turns the Navier-Stokes equations into a system of functional integrals in the trajectory space of a suitable diffusion. In \cite{constantin_iyer}, the authors provide a stochastic representation where the drift term in a stochastic differential equation is implicitly computed as the expected value of an expression involving the flow it drives. This is a step towards a generalization for the Feynman-Kac theory for parabolic equations.

In the current paper, we use the stochastic embedding of Lagrangian systems developped in \cite{note1,cd} to relate the Navier-Stokes, the Euler and the Stokes equations to critical points of a stochastic Lagrangian. An embedding theory is the data of a functional space and an extension of the classical derivative on this set. This allows us to define natural extension of differential operators and differential equations. When dealing with Lagrangian systems an embedding allows us to keep track of the form of the equation but also of the underlying variational principle. Using embedding, we then obtain a strong relation between a classical Lagrangian dynamics and an embedded one, which will be the PDE in our case. 

The body of our paper is arranged as follows: We first remind definitions and results about the Nelson derivatives for stochastic processes and the stochastic embedding of Lagrangian systems developped in \cite{cd}. We then prove that the Navier-Stokes, the Euler and the Stokes equations coincide with extremals of stochastic Lagrangian systems.

\section{Notations and Reminder about Nelson stochastic derivatives and Time reversal}

\subsection{Notations}

Let $T>0$, $\nu>0$ and $d\in\N^*$. Let $\rL$ be the set of all measurable functions $f:\oT\times \R^d\to\R^d$ satisfying the following hypothesis: There exists $K>0$ such that for all $x,y\in\R^d$ :
    $\sup_t \left|f(t,x)-f(t,y)\right|\leq
    K\left|x-y\right|$ et
    $\sup_t \left|f(t,x)\right|\leq
    K(1+\left|x\right|)$.\\

We are given a probability space $(\Omega,\cal{A},\bP)$ on which a family $(W^{(b,\sg)})_{(b,\sg)\in{\rm L}\times\rL}$ of Brownian motions indexed by $\rL\times\rL$ is defined. If
$b,\sg\in\rL$, we denote by $\cal P^{(b,\sg)}$ the natural filtration 
associated to $W^{(b,\sg)}$. Let $\cal P$ be the filtration generated by the filtrations $\cal P^{(b,\sg)}$ where $(b,\sg)\in{\rm L}\times\rL$,
and we set: For $t\in\oT$,
$$\cal P_t=\bigvee_{(b,\sg)\in{\rm L}\times\rL}\cal P_t^{(b,\sg)}.$$

Let $F([0,T]\times \R^d)$ be the space of measurable functions defined on $[0,T]\times \R^d$ and let $F([0,T]\times \Omega)$ be the space of measurable stochastic processes defined on $[0,T]\times \Omega$.

Let us define the involution $\phi: F([0,T]\times \R^d)\to F([0,T]\times \R^d)$ such that 
for all $t\in[0,T]$ and $x\in\R^d$, $(\phi u)(t,x)=-u(T-t,x)$.
We also define the time-reversal involution on stochastic processes:  $r: F([0,T]\times \Omega)\to F([0,T]\times \Omega)$, $r(X)_t(\omega)=X_{T-t}(\omega)$.

It is convenient to use the bar symbol to denote these two involutions. We now agree to denote deterministic functions by small letters and stochastic processes by capital letters. 
So there will not be any confusion when using the bar symbol: $\b u:=\phi u$ and $\b X:=r(X)$.\\

The space $\R^d$ is endowed with its
canonical scalar product $\langle\cdot,\cdot\rangle$. Let $|\cdot|$
be the induced norm. If $f:\oT\times \R^d\to \R$ is a smooth function, we set
$\partial_jf=\frac{\partial f}{\partial x_j}$. We denote by $\nabla
f=(\partial_i f)_i$ the gradient of $f$ and by $\Delta f=\sum_j
\partial_j^2f$ its Laplacian. For a smooth map
$\Phi:\oT\times \R^d\to\R^d$, we denote by $\Phi^j$ its
$j^{th}$-component, by $(\partial_x \Phi)$ its differential which we
represent into the canonical basis of $\R^d$: $(\partial_x
\Phi)=(\partial_j\Phi^i)_{i,j}$, and by $\nabla\cdot \Phi=\sum_j
\partial_j\Phi^j$ its divergence. By convention, we denote by
$\Delta \Phi$ the vector $(\Delta \Phi^j)_j$. The notation $(\Phi\cdot \nabla)\Phi$ denotes the parallel derivative of $\Phi$ along itself, whose coordinates are: $((\Phi\cdot \nabla)\Phi)^i=\sum_j \Phi^j \partial_j \Phi^i$. The image of a vector
$u\in\R^d$ under a linear map $M$ is simply denoted by $Mu$, for
instance $(\partial_x \Phi)u$. A map
$a:\oT\times\R^d\to\R^d\otimes\R^d$ is viewed as a $d\times d$
matrix whose columns are denoted by $a_k$. Finally, we denote by
$\nabla\cdot a$ the vector $(\nabla\cdot a_k)_k$.\\

By the Navier-Stokes {\it equation} we mean the following equation:
\begin{equation}\label{ns}
\partial_t u + (u\cdot \nabla) u =\nu \Delta u -\nabla p,
\end{equation}
where $u:\oT\times \R^d\to\R^d$ is the velocity field and $p:\oT\times \R^d\to \R$ is the pressure.\\
By the Navier-Stokes {\it equations} we mean the following system:
\begin{equation}\label{ns_syst}
\partial_t u + (u\cdot \nabla) u =\nu \Delta u -\nabla p, \quad \nabla\cdot u=0.
\end{equation}
When talking about solutions of (\ref{ns}) (resp. (\ref{ns_syst})), we will always refer to a pair $(u,p)$ of strong solutions of (\ref{ns}) (resp. (\ref{ns_syst})). We distinguish the equation (\ref{ns}) from the system (\ref{ns_syst}) because in the current paper we are interested in the structure of (\ref{ns}). The system  (\ref{ns_syst}) turns out to have the same Lagrangian structure but in that case we have to embed this Lagrangian into a restricted class of diffusions processes, namely those with a divergence free drift.  

We agree to use the same convention for the Euler and the Stokes equations.

\subsection{Time reversal and Nelson derivatives}

\begin{defi}\label{bonnesdiffusions}
We denote by $\Lambda^1$ the space of all diffusions $X$
satisfying the following conditions:
\begin{itemize}
    \item[(i)] $X$ is a solution on $\oT$ of the SDE:
    $dX_t=b(t,X_t)dt+\sigma(t,X_t)dW^{(b,\sg)}_t,\quad X_0=X^0$\\
where $X^0\in L^2(\Omega)$ and $(b,\sg)\in\rL\times\rL$,

    \item[(ii)] For all $t\in (0,T)$, $X_t$ admits a density
    $p_t(\cdot)$,

    \item[(iii)] Setting $a^{ij}=(\sigma\sigma^*)^{ij}$, for all $i\in\{1,\cdots,n\}$, $t_0>0$,
    $$\int_{t_0}^T \int_{\R^d} \left|\partial_j(a^{ij}(t,x)p_t(x))\right|dxdt <
    +\infty,$$
    \item[(iv)] For all $i,j,t$, $$\frac{\partial_j(a^{ij}(t,\cdot)p_t(\cdot))}{p_t(\cdot)}\in {\rm
    L}.$$
\end{itemize}
\end{defi}

A diffusion verifying (i) will be called a $(b,\sg)$-diffusion and we denote it by $X^{(b,\sg)}$.

We denote by $\Lambda^1_v$ the closure of ${\rm Vect}(\Lambda^1)$ in
$L^1(\Omega\times[0,T])$ endowed with the usual norm $\|\cdot\|=E\int
|\cdot|$.

We know that the reversed process of any element $\Lambda^1$ is
still a Brownian diffusion driven by a Brownian motion
$\widehat{W}^{(b,\sg)}$ (\textit{cf} \cite{mns}). We denote
by $\widehat{\cal P}^{(b,\sg)}$ the natural filtration associated to
$\widehat{W}^{(b,\sg)}$. Let $\widehat{\cal P}$ be the filtration
defined by: For all $t\in\oT$,
$$\widehat{\cal P}_t=\bigvee_{(b,\sg)\in{\rm L}\times\rL}\widehat{\cal P}_t^{(b,\sg)}.$$
We finally consider the filtration $\cal F$ such that $\cal
F_t=\widehat{\cal P}_{T-t}$ for all $t\in\oT$.

\begin{pro}[Definition of
Nelson Stochastic derivatives]\label{derives_Nelson}
Let $X=X^{b,\sg}\in\Lambda^1$ with $a^{ij}=(\sigma\sigma^*)^{ij}$
and $a^j=(a^{1j},\cdots,a^{dj})$. For almost all $t\in (0,T)$, the Nelson stochastic derivatives exist in $L^2(\Omega)$ :
\begin{eqnarray}
DX_t: & = & \lim_{h\rightarrow 0^+} E\left[\frac{X_{t+h}-X_t}{h}\mid {\cal
P}_t \right]=b(t,X_t)\\
D_* X_t: & = & \lim_{h\rightarrow 0^+} E\left[\frac{X_t-X_{t-h}}{h}\mid
{\cal F}_t \right]=b(t,X_t)-\frac{1}{p_t(X_t)}\sum_j\partial_j(a^{j}(t,X_t)p_t(X_t)).
\end{eqnarray}
\end{pro}

Therefore, for $X^{b,\sg}\in\Lambda^1$, there exists a measurable function $b_*$ such that $D_*X_t=b_*(t,X_t)$. We call it the left velocity field of $X$. It turns out to be an important objet for the sequel. Also, this field is related to the drift of the time reversed process $\b X$ through the following identity:
\begin{equation}
D \b X= \b {b_*}.
\end{equation}

We denote by $\Lambda^2=\{X\in\Lambda^1; DX,\ D_*X\in\Lambda^1\}$ and
$\Lambda^2_v$ the closure of ${\rm Vect}(\Lambda^2)$ in
$L^1(\Omega\times[0,T])$. We then define $\Lambda_v^k$ in an obvious way.

We denote by $\cal{D}_{\mu}$ the stochastic derivative introduced in
(\cite{note1} Lemme 1.2) and defined by 
\begin{equation}
{\cal D}_\mu = {D+D_* \over 2}
+\mu{D-D_* \over 2}, \quad \mu\in\{0,\pm 1,\pm i\}.
\end{equation}
We can extend $\cal D$ by $\C$-linearity to complex processes
$\Lambda^1_{\C}:=\Lambda_v^1\oplus i\Lambda_v^1$.

\begin{thm}
\label{derivfonc} Let $X^{(b,\sg)}\in \Lambda^1$, $a=\sigma\sigma^*$ and $f\in
C^{1,2}(I\times \R^d)$ such that $\partial_t f$, $\nabla f$ and
$\partial_{ij}f$ are bounded. We get:
\begin{equation}\label{deriv_fonc} 
\mathcal{D}_\mu f(t,X^{(b,\sg)}_t) = \left(\partial_t f +
\mathcal{D} X^{(b,\sg)}_t\cdot \nabla f
+\frac{\mu}{2}\sum_{k,j}a^{kj}\partial_{kj}f\right)(t,X^{(b,\sg)}_t) .
\end{equation}
\end{thm}

Moreover, we can generalize for the operator $\cal D_{\mu}$ the "product
rule" given by Nelson in \cite{ne1} p.$80$, which is a fundamental tool to develop a stochastic calculus of variation:

\begin{lem}
\label{loi_produit} Let $X,Y\in\Lambda_{\C}^1$. Then $
E[\mathcal{D}_\mu X_t\cdot Y_t+X_t\cdot \mathcal{D}_{-\mu}Y_t]=
\frac{d}{dt}E[X_t\cdot Y_t]$.
\end{lem}

The proof is a an immediate consequence of the form of the operator $\cal D$ and to the fact that $\Lambda^1$ is a subspace of the class $S(\cal F,\cal G)$ (\cite{zm} p.226) for which W.
Zheng and P-A. Meyer show the product rule of Nelson (\textit{cf}
\cite{zm} Th. I.$2$ p.$227$).

\section{Reminder about stochastic embedding of Lagrangian systems}

We first define the stochastic analogue of the classical functional.

\begin{defi}
Let $L$ be an admissible Lagrangian function. Set $$
\Xi=\left\{X\in\Lambda^1,E\left[\int_0^T
|L(X_t,\cal{D}_{\mu}X_t)|dt\right]<\infty\right\}.$$ The
functional associated to $L$ is defined by
\begin{equation}
\label{functional} F: \left\{ \begin{array}{ccc} \Xi &
\longrightarrow & \C \\
X & \longmapsto & \di E\left[\int_0^T L(X_t,\mathcal{D}_{\mu}
X_t)dt\right] \end{array}\right.  .
\end{equation}
\end{defi}

In what follows, we need a special notion which we will call
$L$-adaptation:

\begin{defi}
Let $L$ be an admissible Lagrangian function. A process $X\in \Lambda^1$ is said to be $L$-adapted if:
\begin{enumerate}
\item[\rm (i)] $X\in \Xi$;
\item[\rm (ii)] For all $t\in I$, $\di
\partial_x L (X_t,{\cal D}_{\mu} X_t) \in L^2(\Omega)$;
\item[\rm (iii)] $\di \partial_v L (X_t,{\cal D}_{\mu} X_t) \in\Lambda^1$.
\end{enumerate}
\end{defi}

The set of all $L-$adapted processes will be denoted by $\cal L$.\\

We introduce the following terminology:

\begin{defi}
Let $\Gamma$ be a subspace of $\Lambda^1$ and $X\in \Lambda^1$. A
$\Gamma$-variation of $X$ is a stochastic process of the form
$X+Z$, where $Z\in \Gamma$. Moreover set $$
\Gamma_{\Xi}=\left\{Z\in\Gamma, \forall X\in\Xi,
Z+X\in\Xi\right\}.$$
\end{defi}

We now define a notion of {\it differentiable functional}. Let
$\Gamma$ be a subspace of $\Lambda^1$.

\begin{defi}
Let $L$ be an admissible Lagrangian function and $F$ the
associated functional. The functional $F$ is called
$\Gamma$-differentiable at $X\in \cal{L}$ if for all
$Z\in\Gamma_{\Xi}$
\begin{equation}
F(X+Z)-F(X)=dF(X,Z)+R(X,Z) ,
\end{equation}
where $dF(X,Z)$ is a linear functional of $Z\in \Gamma_{\Xi}$
and $R(X,Z)=o(\parallel Z\parallel )$. 

A $\Gamma$-critical process for the functional $F$ is a
stochastic process $X\in \Xi\cap\cal{L}$ such that $dF(X,Z)=0$
for all $Z\in \Gamma_{\Xi}$ such that $Z(a)=Z(b)=0$.
\end{defi}

The main result of \cite{cd} is the following analogue of the least-action
principle for Lagrangian mechanics.

\begin{thm}[Global Least action principle]
\label{glap}
Let $L$ be an admissible lagrangian with all second derivatives
bounded. A necessary and sufficient condition for a process
$X\in \cal L \cap\Lambda^3$ to be a $\Lambda^1$-critical process of the associated functional $F$ is that
it satisfies
\begin{eqnarray}
\label{equaSLAP}
{\partial L\over \partial x}(X_t,\mathcal{D}_{\mu} X_t)-\mathcal{D}_{-\mu}
\left [ \di {\partial L\over \partial v} (X_t,\mathcal{D}_{\mu} X_t) \right ]
=0 .
\end{eqnarray}
We call this equation the Global Stochastic Euler-Lagrange equation (GSEL).
\end{thm}

The main drawback of Theorem \ref{glap}  is that equation (\ref{equaSLAP}) is not {\it coherent} (see \cite{cd},$\S$.6.2) {\it i.e.} that it does not coincide with the direct stochastic embedding of the classical Euler-Lagrange equation of the form
\begin{eqnarray}
{\partial L\over \partial x}(X_t,\mathcal{D}_{\mu} X_t)-\mathcal{D}_{\mu}
\left [ \di {\partial L\over \partial v} (X_t,\mathcal{D}_{\mu} X_t) \right ]
=0 .
\end{eqnarray}
except when $\mu=0$.\\

In order to obtain a coherent embedding without imposing $\mu=0$, we must restrict the set of variations. Let us introduce the space of Nelson differentiable processes:
\begin{equation}
\Nd=\{X\in \Lambda^1, DX=D_*X\}.
\end{equation}
Using $\Nd$-variations we have been able to prove the following result \cite{cd}:

\begin{pro}\label{trivial_Nd}
Let $L$ be an admissible lagrangian with all second derivatives
bounded.
A solution of the equation
\begin{eqnarray}
{\partial L\over \partial x}(X_t,\mathcal{D}_{\mu} X_t)-\mathcal{D}_{\mu}
\left [ \di {\partial L\over \partial v} (X_t,\mathcal{D}_{\mu} X_t) \right ]
=0,
\end{eqnarray}
called the Stochastic Euler-Lagrange Equation (SEL), is a $\Nd$-critical process for the functional
$F$ associated to $L$.
\end{pro}

We have not been able to prove the converse of this lemma for
$\Nd$-variations.

\section{Stochastic embedding and the three equations of fluid mechanics}

We consider the Lagrangian function 
\begin{equation}
\label{lagrangien_nat}
L(t,x,v)=\di \frac{v^2}{2} -p ,
\end{equation}
where $p$ is a function depending on $x$ and $t$. The classical Euler-Lagrange equation yields
\begin{equation}
\di {d^2  x \over d t^2}  =-\nabla p.
\end{equation}

In the sequel, we investigate  the stochastic embedding of this lagrangian over diffusion processes of specific forms. Precisely, the Navier-Stokes equation, the Euler equation and the Stokes equation turn out to be respectively related to the following subspaces of $\Lambda^1$:
\begin{enumerate}
\item The class $\Lambda^2_c$ composed of all the diffusions $X^{(u,\sg)}$ such that $\sg$ is constant, the drift $u$ is $C^2$, bounded with all its second derivatives bounded, and $\nabla \log \rho_t$ has bounded second order derivatives. We can prove that $\Lambda^2_c\subset \Lambda^2$ thanks to Prop. 2 p.394 in \cite{dn}. Recall that a $(u,\sigma)$-diffusion with a constant diffusion coefficient $\sg$ is of the form:
\begin{equation}
X_t = X_0+\int_0^t u(s,X_s) ds + \sigma W_t.
\end{equation}
\item The class $\Lambda^2_0$ composed of all the smooth processes of the form:
\begin{equation}
X_t = X_0+\int_0^t u(s,X_s) ds,
\end{equation}
where $u\in C^2(\R^+,\R^d)$ and $X_0\in L^2(\Omega)$.
\item The class $\Lambda^2_{B}$ composed of all the processes of the form:
\begin{equation}
X_t = X_0+\int_0^Tu(s,\sigma B_s) \, ds,
\end{equation}
where $u\in C^2(\R^+,\R^d)$, $\sigma\in \R$, $X_0\in L^2(\Omega)$ and $B$ is a Brownian bridge pinned to be $0$ at $T$.
\end{enumerate}

\section{The incompressible Navier-Stokes equation}

Recall that the incompressible Navier-Stokes equation can be written as
\begin{equation}
\partial_t u + (u\cdot \nabla) u =\nu \Delta u -\nabla p .
\end{equation}

Regarding stochastic least action principles, we can consider two cases: The first one uses a full set of variation and the second one uses $\Nd$-variations.

\subsection{Full variations}

Applying Th. \ref{glap}, a $(u,\sg)$-diffusion $X$ which is a critical point of the natural lagrangian (\ref{lagrangien_nat}) for full variations, satisfies
\begin{equation}
\label{stocbase}
{\cal D}_{-\mu} \left ( {\cal D}_{\mu} X_t \right ) =-\nabla p.
\end{equation}

When $\mu=1$, i.e. ${\cal D} =D$, we obtain $D_* (DX_t )=D_* u(t,X_t) = -\nabla p$ and finally:
\begin{equation}
\partial_t u + u_* \cdot \nabla u -\di {\sigma^2 \over 2} \Delta u  = - \nabla p.
\end{equation}

When $\mu=-1$, i.e. ${\cal D} =D_*$, we obtain $D (D_* X_t )=D u_* (t,X_t )= -\nabla p$ and finally:
\begin{equation}
\partial_t u_*  + u \cdot \nabla u_*  +\di {\sigma^2 \over 2} \Delta u_*   =  -\nabla p.
\end{equation}

We notice that the field $u$ does not satisfy the Navier-Stokes equations when using full variations in the case $\mu=1$ since the constant in front of $\Delta u$ is positive. When $\mu=-1$, we can obtain the Navier-Stokes equation if and only if $u_*=u$, i.e. if $DX_t =D_* X_t$ which corresponds to Nelson differentiable processes. These processes can be described through a given partial differential equation (see \cite{cd}). In particular, a diffusion belonging to $\Lambda^1$ with a constant diffusion coefficient cannot be a Nelson differentiable process. As a consequence, we have no hope to recover the Navier-Stokes equation into this setting.

\subsection{$\Nd$-variations}
 
One of the main results of \cite{cd} states that the stochastic Euler-Lagrange equation reads in that case:
\begin{equation}
{\cal D}_{\mu} \left ( {\cal D}_{\mu} X_t \right ) =-\nabla p .
\end{equation}

When $\mu=1$, i.e. ${\cal D} =D$, we obtain $D^2X_t =D u(t,X_t )=-\nabla p$ and finally:

\begin{equation}
\partial_t u  + u \cdot \nabla u  +\di {\sigma^2 \over 2} \Delta u   =  -\nabla p.
\end{equation}

We notice that up to the sign of the constant term we obtain the form of the Navier-Stokes equation. As $\sigma^2$ is always positive we have a apparent obstruction to obtain the Navier-Stokes equation using the $D$-embedding. However this obstruction can be removed considering drifts of the form $\b u$.

When $\mu=-1$, i.e. ${\cal D}=D_*$, we can obtain the Navier Stokes equation from a $(u,\sg)$-diffusion and it will be satisfied by $u_*$.
 
These results are the content of the following theorem:

\begin{thm}\label{thm_navier_temp}
Let $\nu>0$ and  $X^{(u,\sqrt{2\nu})} \in \Lambda_c^2$. \\
If the left velocity field $u_* (t,x)$ of $X$ satisfies the Navier-Stokes equation
\begin{equation}\label{navieru*}
\partial_t u_*  + u_* \cdot \nabla u_*  -\nu \Delta u_*   =  -\nabla p ,
\end{equation}
then the stochastic process $X^{(u,\sqrt{2\nu})} $ is a $\Nd$-critical process of the stochastic functional 
\begin{equation}
X\mapsto \mbox{\rm E} \left [ \int_0^T L(X_t ,D_* X_t ) dt \right],
\end{equation}
where $L$ is the natural lagrangian $L(x,v)= \frac{v^2}{2} -p$.

Moreover, if $u$ verifies the Navier-Stokes equation then the stochastic process $X^{(\b u,\sqrt{2\nu})} $ is a $\Nd$-critical process of the stochastic functional 
\begin{equation}
X\mapsto \mbox{\rm E} \left [ \int_0^T L(X_t ,D X_t ) dt \right],
\end{equation}
where $L$ is the natural lagrangian $L(x,v)= \frac{v^2}{2} + \b p$.
\end{thm}

\begin{proof}
Let $X^{(u,\sigma)} \in \Lambda_c^2$ and let $u_*$ be its left velocity field $D_*X$.
Then we have:
\begin{equation}
D_* (D_* X^{(u,\sigma)}_t )=\left(\partial_t u_*  + u_* \cdot\nabla u_*  -\di {\sigma^2 \over 2} \Delta u_*\right)(t,X^{(u,\sigma)}).
\end{equation}
Therefore if $u_* (t,x)$ satisfies the Navier-Stokes equation (\ref{navieru*}), then:
$$D_* (D_* X^{(u,\sqrt{2\nu})}_t )=-\nabla p(t,X^{(u,\sigma)}),$$
so $X^{(u,\sqrt{2\nu})} $ is a $\Nd$ critical process of the stochastic functional 
\begin{equation}
X\mapsto \mbox{\rm E} \left [ \int_0^T L(X_t ,D_* X_t ) dt \right]. 
\end{equation}

Now, assume that $u$ satisfies the Navier-Stokes equation and let us consider the diffusion $X:=X^{(\b u,\sg)}$. Then 
\begin{eqnarray}
D^2X_t & = & (\partial_t\b u+\b u\cdot \nabla \b u +\nu \Delta \b u) (t,X_t)=(\partial_t u+u\cdot \nabla u -\nu \Delta  u) (T-t,X_t)\\
 & = & -\nabla p(T-t,X_t)=\nabla \b p(t,X_t),
\end{eqnarray}
which yields the last statement of the Theorem. 

\end{proof}

\begin{rema}\rm
Due to the fact that Theorem \ref{glap} does not give an equivalence between $\Nd$-critical process and the stochastic Euler-Lagrange equation, we are not able to prove that the Navier-Stokes can be obtain by a coherent $D_*$ embedding procedure, i.e. that the following diagramm commutes
\begin{eqnarray}
\xymatrix{
  & L(x(t),\dot{x} (t) ) \ar[d]_{\mbox{\rm LAP}} \ar[r]^{\mbox{\rm Emb} (D_* )} & L(X_t, D_* X_t)
  \ar[d]^{\Nd-\mbox{\rm Stochastic LAP }}       \\
  & \di {d\over dt} {\partial L\over \partial v} (z(t))   = {\partial L\over \partial x} (z(t))\  \ar[r]_{\quad \rm Emb(D_* )\quad }  &  \ D_*  {\partial L\over \partial v} (Z(t))  ={\partial L\over \partial x}  (Z(t))  }
\end{eqnarray}
where $z(t)=(x(t),\dot{x}(t))$, $Z(t)=(X_t,D_* X_t )$, and $LAP$ stands for "Least Action Principle".
\end{rema}

\section{About the incompressible Euler equation}

The incompressible Euler equation 
\begin{equation}
\partial_t u + (u\cdot \nabla) u =-\nabla p 
\end{equation}
is of course obtained as a limit of the Navier-Stokes equation when the viscosity $\nu$ goes to zero. However, in our case this has a strong consequence on the underlying embedding from the functional point of view. Indeed, when $\nu$ goes to zero, the set ${\cal S}_0$ reduces to deterministic stochastic processes which are sufficiently regular. On the contrary the basic property of diffusion processes is that each trajectory is a.s. H\"olderian with exponent $1/2 -\epsilon$ for $\epsilon >0$. 

First of all, let us remark that $D=D_*=\frac{d}{dt}$ on $\cal S_0$.

\begin{cor}\label{euler_rep}
A process $X_t \in {\cal S}_0$ of the form 
\begin{equation}
dX_t =u(t,X_t )dt
\end{equation}
is a $\Nd$ or $\Lambda^1$-critical process of the stochastic functional 
\begin{equation}
X\mapsto \mbox{\rm E} \left [ \int_0^T L(X_t ,D X_t ) dt \right ] 
\end{equation}
where $L(x,v)=\di \frac{v^2}{2} -p$, if and only if the velocity field $u(t,x)$ satisfies the Euler equation
\begin{equation}
\partial_t u  + u\cdot\nabla u = -\nabla p.
\end{equation}
\end{cor}

One must be carrefull with this result. Even if the basic underlying set of stochastic processes corresponds to deterministic processes, the stochastic calculus of variations embed these deterministic processes into arbitrary diffusion processes.

\section{About the incompressible Stokes equation}

Through processes in the class $\Lambda_B$, we can find the following stochastic least action formulation of the Stokes equation:
\begin{thm}\label{stokes_rep}
Let $X \in \Lambda_B$ be a process of the form:
\begin{equation}
dX_t =u(t,\sqrt{2\nu } B_t )dt.
\end{equation}
The velocity field $u$ satisfies the Stokes equation
\begin{equation}
\partial_t u  -\nu \Delta u   = - \nabla p ,
\end{equation}
if and only if the stochastic process $X_t$ is a $\Lambda^1$ critical process of the stochastic functional 
\begin{equation}
X\mapsto \mbox{\rm E} \left [ \int_0^T L(\sqrt{2\nu } B_t ,D X_t ) dt \right ] 
\end{equation}
where $L(x,v)=\di \frac{v^2}{2} -p$.
\end{thm}

\begin{proof}
The $\Lambda^1$ least action principle reads:
$$D_*DX_t=-\nabla p(\sqrt{2\nu } B_t ).$$
But 
\begin{equation}
D_*DX_t= \left( \partial_t u +\sqrt{2\nu }\  (\partial_x u)D_*B_t -\nu \Delta u \right)(t,\sqrt{2\nu } B_t).
\end{equation}
Since $D_*B=0$ and the probability density of the bridge is everywhere positive, we can deduce the desired equivalence. 

\end{proof}

\section{The temperature equations associated to the various fluid equations}

\subsection{Temperature for the Navier-Stokes equation}

As mentioned in \cite{villani}, one has to associate to the Navier Stokes equations an equation for the temperature field $\theta$. This leads to the so-called Navier-Stokes-Fourier system (see e.g. \cite{golse}):
\begin{eqnarray}\label{navier_syst}
\partial_t b  + b \cdot \nabla b  -\nu \Delta b  & = & -\nabla p, \qquad \nabla \cdot b = 0,\\
\partial_t \theta  +b\cdot \nabla \theta & = & \kappa \Delta \theta. \label{eq_temp}
\end{eqnarray}
The coefficient $\kappa$ is a priori different from the viscosity coefficient $\nu$.

We have $\nabla\cdot (\rho u)=u\cdot \nabla \rho+ \rho \nabla \cdot u$. So the incompressibility condition $\nabla\cdot b=0$ implies that the temperature equation (\ref{eq_temp}) is actually the Fokker Planck equation:
\begin{equation}\label{FK}
\partial_t \rho  =-\nabla\cdot (\rho b) + \kappa \Delta \rho,
\end{equation}
satisfied by the probability density $\rho$ of a diffusion $X^{(b,\sqrt{2\kappa})}$.

Therefore the Navier Stokes system together with the temperature equation can be described by the two following diffusions: 
\begin{itemize}
\item[$\bullet$] $X^{(u,\sqrt{2\nu})}$, which is the critical point of the natural Lagrangian $L(x,v)=\frac{v^2}{2} -p$, and
\item[$\bullet$] $X^{(u_*,\sqrt{2\kappa})}$ whose density satisfies the associated temperature equation.
\end{itemize}

\subsection{Temperature for the Stokes equation}

The Stokes equation 
\begin{equation}
\partial_t u - \nu \Delta u   = - \nabla p ,
\end{equation}
differs from the Navier-Stokes and Euler equations in the sense that it cannot be obtained as any limit of the previous ones. Moreover, the related equation on the temperature $\theta$ does not surprisingly depends on the velocity field $u$:
\begin{equation}\label{heat}
\partial_t \theta  =\kappa \Delta \theta.
\end{equation}

Based on Theorem \ref{thm_navier_temp}, this specific structure might suggest to uncouple the underlying process of the velocity field and the velocity field itself. Based on this remark, we came up to Theorem \ref{stokes_rep}. Therefore the Navier Stokes system together with the temperature equation can be described by the two following processes: 
\begin{itemize}
\item[$\bullet$] $X_t =X_0+\int_0^T u(t,\sqrt{2\nu} B_t )dt$, which is the critical point of the natural Lagrangian $L(x,v)=\frac{v^2}{2} -p$, and 
\item[$\bullet$] $\sqrt{2\kappa}\ \b B$ which is a Brownian motion and whose density satisfies the associated heat equation, that is the temperature equation (\ref{heat}).
\end{itemize}

\subsection{Temperature for the Euler equation}

Corollary \ref{euler_rep} naturally leads to the temperature equation for the Euler equation, thanks to elementary results about the transport of measure through an ordinary differential equation.

\end{document}